\documentclass[11pt]{amsart}

\usepackage[a4paper,margin=31mm]{geometry}
\usepackage[T1]{fontenc}
\usepackage[utf8]{inputenc}
\usepackage{lmodern}
\usepackage{microtype}
\usepackage{amsmath,amssymb}
\usepackage{hyperref}

\hypersetup{
  colorlinks=true,
  linkcolor=black,
  citecolor=black,
  urlcolor=black,
  pdftitle={A density deficit for sums of three cube-full numbers},
  pdfsubject={Additive representations by cube-full numbers},
  pdfkeywords={cube-full numbers, powerful numbers, additive representations,
  cubic characters, Chebotarev density theorem}
}

\newtheorem{theorem}{Theorem}[section]
\newtheorem{proposition}[theorem]{Proposition}
\newtheorem{lemma}[theorem]{Lemma}
\newtheorem{corollary}[theorem]{Corollary}

\newcommand{\N}{\mathbb N}
\newcommand{\F}{\mathcal F}
\newcommand{\C}{\mathcal C}
\newcommand{\R}{\mathcal R}
\newcommand{\Fq}{\mathbb F_q}
\newcommand{\vol}{\operatorname{vol}}

\title[A density deficit for three cube-full summands]
{A Density Deficit for Sums of Three Cube-Full Numbers}
\author{Basile Beyer de Ryke}
\address{Independent researcher, Brussels, Belgium}
\email{basile.beyer@gmail.com}

\date{Revised 26 July 2026}

\begin{document}

\begin{abstract}
Let $\F$ be the set of positive cube-full integers.  We prove that, for
every $\varepsilon>0$, there is a reduced residue class in which
$\F+\F+\F$ has upper relative density at most $\varepsilon$.
Consequently, the positive integers not representable as sums of at most
three cube-full numbers have positive lower natural density.  This proves the
infinitude assertion in Erd\H{o}s Problem \#940 for $r=3$.  The argument
combines a cubic-character estimate with a truncation of the parametrisation
by the canonical 4-full part.
\end{abstract}

\maketitle

\section{Introduction}

For an integer $r\geq2$, a positive integer $n$ is called $r$-full if
$p\mid n$ implies $p^r\mid n$.  Erd\H{o}s asked whether, for each
$r\geq3$, infinitely many positive integers fail to be sums of at most
$r$ $r$-full numbers, and whether the integers which do admit such a
representation have density zero \cite[p.~33]{Erdos1976}.  The problem was
later recorded by Erd\H{o}s and Ivi\'c in the Oberwolfach Problem Book and is
now listed as Erd\H{o}s Problem \#940 \cite{Bloom940}.  Erd\H{o}s attributed
the infinitude assertion to a simple counting argument, but Schinzel observed
that the argument was incorrect.  The case of three cube-full summands was
raised again by Thomas Bloom in the problem session of the 2025 Oberwolfach
workshop on analytic number theory
\cite[pp.~2756--2757]{Oberwolfach2025}.

We write $\N=\{1,2,\ldots\}$, so that all summands below are positive, and
we regard $1$ as cube-full.  Let $\F$ denote the set of cube-full positive
integers and put
\[
  \R_3=\F+\F+\F.
\]
For $A\subseteq\N$, a residue class $a\pmod M$, and $X\geq1$, set
\[
  A(X;a,M)=\#\{n\leq X:n\in A,\ n\equiv a\pmod M\}.
\]

\begin{theorem}\label{thm:main}
For every $\varepsilon>0$, there are integers $M\geq1$ and $a$, with
$(a,M)=1$, such that
\[
  \limsup_{X\to\infty}\frac{M}{X}\,\R_3(X;a,M)
  \leq \varepsilon.
\]
\end{theorem}

\begin{corollary}\label{cor:density}
The complement of $\R_3$ has positive lower natural density.  In
particular, infinitely many positive integers are not sums of three
cube-full numbers.
\end{corollary}

\begin{proof}
Apply Theorem~\ref{thm:main} with $\varepsilon=1/2$.  Since the progression
$a\pmod M$ contains $X/M+O(1)$ positive integers not exceeding $X$,
\[
 \liminf_{X\to\infty}\frac1X\#\{n\leq X:n\notin\R_3\}
 \geq \frac{1}{2M}>0.
\]
\end{proof}

Equation~\eqref{eq:Fcount} below shows that one-term and two-term sums
contribute only $O(X^{2/3})$ integers up to $X$.  Hence the integers not
representable as sums of at most three cube-full numbers also have positive
lower natural density.  In particular, Corollary~\ref{cor:density} confirms
Bloom's prediction in the 2025 Oberwolfach problem session: the set of
three-term sums has upper natural density strictly less than $1$
\cite[pp.~2756--2757]{Oberwolfach2025}.

The main arithmetic input is the following. For a finite set $D$ of 4-full parts, let $\zeta_3$ be a primitive cube
root of unity and consider
\[
 K_D=\mathbb Q\bigl(\zeta_3,\sqrt[3]{d}:d\in D\bigr).
\]
This is the splitting field of
\[
 (X^2+X+1)\prod_{d\in D}(X^3-d).
\]
By the Chebotarev density theorem, the primes that split completely in
$K_D$ have positive density \cite[Chapter~VII,
Theorem~13.4]{Neukirch1999}.  After removing the finitely many primes dividing
$3\prod_{d\in D}d$, every remaining prime $q$ satisfies
$q\equiv1\pmod3$, and every $d\in D$ is a non-zero cube modulo $q$.
Lemma~\ref{lem:splitting} gives the short proof of these two conclusions.

\paragraph*{Outline of the paper.}
Every cube-full integer has a unique expression $dx^3$, where $d$ is
its 4-full part.  The two sums
\[
 W=\sum_d d^{-1/3}
 \qquad\text{and}\qquad
 H=\sum_d3^{\omega(d)}d^{-1/3},
\]
taken over all 4-full parts $d$, converge; here $\omega(d)$ is the
number of distinct prime divisors of $d$.  We begin by choosing a finite
set $D$ so that the tail of $H$ outside $D$ is small.  Chebotarev's
theorem then gives a positive-density set of primes $q\equiv1\pmod3$ for
which every $d\in D$ is a cube modulo $q$.  The sum of the reciprocals
of these primes diverges, so we may select sufficiently many of them.

For each selected prime $q$, we choose a non-cube residue $u_q$.  The
Chinese remainder theorem combines these residues into one reduced residue
class $a\pmod M$.  If $d_1,d_2,d_3\in D$, then the congruence
\[
 d_1x_1^3+d_2x_2^3+d_3x_3^3\equiv a\pmod q
\]
has at most $q^2-2q$ solutions.  Thus each selected prime gives a local
density factor at most $1-2/q$, and the product of these factors can be
made arbitrarily small.

The finite set $D$ does not contain every 4-full part.  After $M$ has
been fixed, we therefore choose a larger finite set $E$.  The tail of $H$
bounds the representations whose three 4-full parts belong to $E$ but do
not all belong to $D$.  The tail of $W$ bounds the representations having
at least one 4-full part outside $E$.  Combining these three estimates
proves Theorem~\ref{thm:main}.

We begin in Section~2 with the canonical 4-full part.  We then prove the
local cubic-character estimate in Section~3 and count representations with
fixed 4-full parts in Section~4.  Section~5 combines these ingredients to
prove the main theorem.  Finally, Section~6 explains why the elementary
counting argument alone cannot give the theorem.

The other question in the Oberwolfach report, whether sums of three
cube-full numbers have positive density, remains open.  A positive answer
would follow from the corresponding conjecture for three positive cubes,
which is itself unresolved; see Maynard \cite{Maynard2026}.

\section{Cube-full numbers and their 4-full parts}

We begin by separating a perfect cube from its 4-full part.  The resulting
weights over the 4-full parts are summable.

\begin{lemma}\label{lem:canonical}
Every cube-full positive integer $n$ has a unique expression
\[
  n=x^3b^4c^5,
\]
where $x,b,c\geq1$ and $bc$ is squarefree.
\end{lemma}

\begin{proof}
Let $p^e\Vert n$.  If $e>0$, then $e\geq3$.  According as
$e\equiv0,1,2\pmod3$, there is a unique expression
\[
  e=3u,\qquad e=3u+4,\qquad e=3u+5,
\]
respectively, with $u\geq0$.  In the second case place $p$ in $b$,
and in the third place it in $c$; the remaining power of $p$ belongs to
$x^3$.  Doing this prime by prime proves both existence and uniqueness.
\end{proof}

This decomposition is standard in the study of cube-full numbers; compare
\cite{Shiu1991}.

Let
\[
  \C=\{b^4c^5:bc\text{ is squarefree}\}.
\]
Thus $\C$ is the set of possible 4-full parts, and every $m\in\F$
is uniquely $m=dx^3$, with $d\in\C$ and
$x\in\N$.  Here and below, $\omega(d)$ denotes the number of distinct
prime divisors of $d$.  For later use, set
\begin{align}
 W&=\sum_{d\in\C}d^{-1/3}
   =\prod_p(1+p^{-4/3}+p^{-5/3}),\label{eq:W}\\
 H&=\sum_{d\in\C}3^{\omega(d)}d^{-1/3}
   =\prod_p(1+3p^{-4/3}+3p^{-5/3}),\label{eq:H}
\end{align}
and both products converge absolutely.  For a finite set
$D\subset\C$, write
\[
  W_D=\sum_{d\in D}d^{-1/3},\qquad
  H_D=\sum_{d\in D}3^{\omega(d)}d^{-1/3}.
\]
From \eqref{eq:W} we also obtain
\begin{equation}\label{eq:Fcount}
  \#(\F\cap[1,X])\leq W X^{1/3},
\end{equation}
since a fixed 4-full part $d$ allows at most $(X/d)^{1/3}$ choices for $x$.

\section{A cubic-character estimate}

We next prove the local estimate that gives a saving at each selected prime.
The saving is of order $q^{-1}$.

Let $q\equiv1\pmod3$ be prime.  Fix a non-trivial cubic character $\chi$
on $\Fq^\times$, extended by $\chi(0)=0$, and write
\[
 e_q(t)=\exp(2\pi i t/q),\qquad
 \tau(\chi)=\sum_{y\in\Fq}\chi(y)e_q(y).
\]
Since $q\equiv1\pmod6$, we have $\chi(-1)=1$, and therefore
\[
 \tau(\chi)\tau(\overline\chi)=q,
 \qquad |\tau(\chi)|=|\tau(\overline\chi)|=q^{1/2}.
\]

\begin{lemma}\label{lem:three-cubes}
For $u\in\Fq^\times$, let
\[
 N_q(u)=\#\{(x,y,z)\in\Fq^3:x^3+y^3+z^3=u\}.
\]
Then
\begin{equation}\label{eq:exact-local}
 N_q(u)=q^2+3q\bigl(\chi(u)+\overline\chi(u)\bigr)
 -\frac{\tau(\chi)^3+\tau(\overline\chi)^3}{q}.
\end{equation}
In particular, if $u$ is not a cube modulo $q$, then
\begin{equation}\label{eq:deficit}
  q^{-2}N_q(u)\leq1-\frac2q.
\end{equation}
\end{lemma}

\begin{proof}
For $s\neq0$, put $S(s)=\sum_{x\in\Fq}e_q(sx^3)$.  The number of
cube roots of $y\in\Fq$ is
$1+\chi(y)+\overline\chi(y)$, whence
\begin{equation}\label{eq:S}
 S(s)=\tau(\chi)\overline\chi(s)
      +\tau(\overline\chi)\chi(s).
\end{equation}
Additive orthogonality gives
\[
 N_q(u)=\frac1q\sum_{s\in\Fq}e_q(-su)S(s)^3,
 \qquad S(0)=q.
\]
Writing $\tau=\tau(\chi)$, $\tau'=\tau(\overline\chi)$, and using
$\tau\tau'=q$, we obtain, for $s\neq0$,
\[
 S(s)^3=\tau^3+(\tau')^3
       +3q\tau\overline\chi(s)+3q\tau'\chi(s).
\]
Now
\[
 \sum_{s\neq0}e_q(-su)=-1,
\]
while the two remaining sums are the usual Gauss sums
\[
 \sum_{s\neq0}\overline\chi(s)e_q(-su)=\chi(u)\tau',
 \quad
 \sum_{s\neq0}\chi(s)e_q(-su)=\overline\chi(u)\tau.
\]
Substitution proves \eqref{eq:exact-local}.  If $u$ is a non-cube, then
$\chi(u)+\overline\chi(u)=-1$, so
\[
 N_q(u)\leq q^2-3q+2q^{1/2}\leq q^2-2q;
\]
here $q\geq7$.  This is \eqref{eq:deficit}.
\end{proof}

For integer coefficients $a_1,a_2,a_3$, define
\[
 N_q(u;\mathbf a)=
 \#\{\mathbf x\in\Fq^3:a_1x_1^3+a_2x_2^3+a_3x_3^3=u\},
 \qquad \rho_q(u;\mathbf a)=q^{-2}N_q(u;\mathbf a).
\]

When a coefficient is not controlled by the chosen splitting conditions, we
use only the following uniform upper bound.

\begin{lemma}\label{lem:uniform}
Let $u\in\Fq^\times$.  If $q\nmid a_1a_2a_3$, then
\[
  |N_q(u;\mathbf a)-q^2|\leq8q.
\]
For arbitrary coefficients, $N_q(u;\mathbf a)\leq3q^2$.  Consequently,
\begin{equation}\label{eq:weighted-local}
 \rho_q(u;\mathbf a)
 \leq\left(1+\frac8q\right)
 3^{\sum_{j=1}^3\mathbf1_{q\mid a_j}}.
\end{equation}
\end{lemma}

\begin{proof}
When all coefficients are non-zero modulo $q$, orthogonality and
\eqref{eq:S} give
\[
 N_q(u;\mathbf a)-q^2
 =\frac1q\sum_{s\neq0}e_q(-su)\prod_{j=1}^3S(a_js).
\]
Expanding the product gives eight terms.  Each coefficient has modulus
$q^{3/2}$, and the corresponding sum over $s$ is either $-1$ or a
Gauss sum of modulus $q^{1/2}$.  After division by $q$, each term has
modulus at most $q$.

If exactly one coefficient is non-zero, the active cubic equation has at most
three solutions and the other two variables are free.  If exactly two are
non-zero, we fix one active variable and again have at most three choices for
the other, while the inactive variable is free.  Both cases give $3q^2$.
There is no solution when all coefficients vanish because $u\neq0$; with
three non-zero coefficients, the first estimate and $q\geq7$ give
$q^2+8q<3q^2$.  Finally, if some $a_j$ vanishes modulo $q$, the right
side of \eqref{eq:weighted-local} is at least $3$, and otherwise the first
estimate applies.  This proves \eqref{eq:weighted-local}.
\end{proof}

\section{Counting with fixed 4-full parts}

We now use the local estimates to count representations once the three
4-full parts have been fixed.

Set
\[
 V=\vol\{\mathbf t\in\mathbb R_{\geq0}^3:
 t_1^3+t_2^3+t_3^3\leq1\}=\Gamma(4/3)^3.
\]
Let $M$ be squarefree, let $(a,M)=1$, and fix
$\mathbf d=(d_1,d_2,d_3)\in\C^3$.  We denote by
$R_{\mathbf d}(X;a,M)$ the number of $\mathbf x\in\N^3$ for which
\[
 d_1x_1^3+d_2x_2^3+d_3x_3^3\leq X,
 \qquad
 d_1x_1^3+d_2x_2^3+d_3x_3^3\equiv a\pmod M.
\]

\begin{proposition}\label{prop:lattice}
Suppose $M=\prod_{q\in Q}q$, where the primes in $Q$ are distinct.  For
fixed $M$ and $\mathbf d$,
\begin{equation}\label{eq:lattice}
 R_{\mathbf d}(X;a,M)
 =\frac{VX}{M}(d_1d_2d_3)^{-1/3}
   \prod_{q\in Q}\rho_q(a;\mathbf d)
   +O_{M,\mathbf d}(X^{2/3}).
\end{equation}
\end{proposition}

\begin{proof}
The region
\[
 \Omega_{\mathbf d}(X)=\{\mathbf t\in\mathbb R_{\geq0}^3:
 d_1t_1^3+d_2t_2^3+d_3t_3^3\leq X\}
\]
is a fixed compact semialgebraic region dilated by $X^{1/3}$, and its
volume is $VX(d_1d_2d_3)^{-1/3}$.  Davenport's lattice-point principle
\cite{Davenport1951}, applied after translating and rescaling each residue
class modulo $M$, yields
\[
 \#\{\mathbf x\in\Omega_{\mathbf d}(X)\cap\N^3:
       \mathbf x\equiv\mathbf r\pmod M\}
 =\frac{VX}{M^3}(d_1d_2d_3)^{-1/3}
  +O_{M,\mathbf d}(X^{2/3}).
\]
The body and all its coordinate projections are convex, so their intersections
with coordinate-parallel lines are intervals.  Its two-dimensional coordinate
projections have area $O_{\mathbf d}(X^{2/3})$, its one-dimensional
projections have length $O_{\mathbf d}(X^{1/3})$, and the zero-dimensional
term is $O(1)$.  These are precisely the quantities in Davenport's error
term; the finitely many residue classes are absorbed into the dependence on
$M$.
Replacing non-negative lattice points by positive ones affects only points on
the coordinate hyperplanes and is absorbed by the same error term.

The Chinese remainder theorem shows that the number of admissible residue
vectors $\mathbf r\pmod M$ is
\[
 \prod_{q\in Q}N_q(a;\mathbf d)
 =M^2\prod_{q\in Q}\rho_q(a;\mathbf d).
\]
There are only finitely many such vectors once $M$ is fixed.  Summing the
preceding estimate proves \eqref{eq:lattice}.
\end{proof}

Keeping the 4-full parts fixed is essential.  The dependence of the error in
\eqref{eq:lattice} on $\mathbf d$ forces us to restrict the 4-full parts before
applying Proposition~\ref{prop:lattice}.  The omitted triples will be
estimated without congruence information.

\section{Proof of the main theorem}

We first justify the splitting-prime input stated in the introduction.

\begin{lemma}\label{lem:splitting}
Let $D$ be a finite set of positive integers, let $\zeta_3$ be a
primitive cube root of unity, and put
\[
 K_D=\mathbb Q\bigl(\zeta_3,\sqrt[3]{d}:d\in D\bigr).
\]
Then $K_D/\mathbb Q$ is a finite Galois extension.  If
$q\nmid 3\prod_{d\in D}d$ is a prime that splits completely in $K_D$,
then $q\equiv1\pmod3$, and every $d\in D$ is a non-zero cube modulo
$q$.
\end{lemma}

\begin{proof}
The field $K_D$ is the splitting field over $\mathbb Q$ of
\[
 (X^2+X+1)\prod_{d\in D}(X^3-d),
\]
so it is finite and Galois.  Let $\mathfrak q$ be a prime of $K_D$
above $q$.  Since $q$ splits completely, its residue degree at
$\mathfrak q$ is one, and hence
$\mathcal O_{K_D}/\mathfrak q\cong\mathbb F_q$.

The reduction of $\zeta_3$ satisfies $z^2+z+1=0$ and is not $1$,
because $3\not\equiv0\pmod q$.  It therefore has order $3$ in
$\mathbb F_q^\times$, so $3\mid q-1$.  For each $d\in D$, the
algebraic integer
$\sqrt[3]{d}$ reduces to an element whose cube is $d$.  Since $q\nmid
d$, this element is non-zero.  Thus $d\in(\mathbb F_q^\times)^3$.
\end{proof}

\begin{proof}[Proof of Theorem~\ref{thm:main}]
It is enough to treat $0<\varepsilon<1$.  Put
$\delta=\varepsilon/4$.  Choose $T>0$ so that
\begin{equation}\label{eq:T}
  VW^3e^{-2T}<\delta.
\end{equation}
Since $H<\infty$, there is a finite set $D\subset\C$ satisfying
\begin{equation}\label{eq:D}
  Ve^{8(T+1)}(H^3-H_D^3)<\delta.
\end{equation}

Let
\[
 K=\mathbb Q\bigl(\zeta_3,\sqrt[3]{d}:d\in D\bigr).
\]
By Lemma~\ref{lem:splitting}, this is a finite Galois extension, and every
prime $q\nmid3\prod_{d\in D}d$ that splits completely in $K$ satisfies
$q\equiv1\pmod3$, and every $d\in D$ is a non-zero cube modulo $q$.
By the Chebotarev density theorem
\cite[Chapter~VII, Theorem~13.4]{Neukirch1999}, the primes that split
completely in $K$ have Dirichlet density $1/[K:\mathbb Q]>0$.
Removing the finitely many primes dividing $3\prod_{d\in D}d$ does not
change this density.  In particular, the sum of the reciprocals of the
remaining primes diverges.  We may consequently choose a finite set $Q$
of them for which
\begin{equation}\label{eq:Q}
  T\leq\sum_{q\in Q}\frac1q\leq T+1.
\end{equation}
The upper bound follows by adding the primes one at a time, since $q\geq7$.

For each $q\in Q$, choose a non-cube $u_q\in\Fq^\times$.  Put
$M=\prod_{q\in Q}q$, and choose by the Chinese remainder theorem a class
$a\pmod M$ such that $a\equiv u_q\pmod q$ for every $q\in Q$.  Thus
$(a,M)=1$.  Finally, choose a finite set $E\subset\C$ containing $D$
and so large that
\begin{equation}\label{eq:E}
  M(W^3-W_E^3)<\delta.
\end{equation}

We count ordered representations by exactly three cube-full integers.  First
suppose $\mathbf d\in D^3$.  Every $d_j$ is a non-zero cube modulo every
$q\in Q$, so a change of variables and Lemma~\ref{lem:three-cubes} give
\[
 \rho_q(a;\mathbf d)\leq1-\frac2q.
\]
It follows from \eqref{eq:Q} that
\[
 \prod_{q\in Q}\rho_q(a;\mathbf d)
 \leq\exp\left(-2\sum_{q\in Q}\frac1q\right)\leq e^{-2T}.
\]
Summing Proposition~\ref{prop:lattice} over $D^3$, and using
\eqref{eq:T}, we obtain
\begin{equation}\label{eq:small}
 \sum_{\mathbf d\in D^3}R_{\mathbf d}(X;a,M)
 \leq\delta\frac XM+O_{M,D}(X^{2/3}).
\end{equation}

For $\mathbf d\in E^3\setminus D^3$, Lemma~\ref{lem:uniform} and
\eqref{eq:Q} instead give
\begin{align*}
 \prod_{q\in Q}\rho_q(a;\mathbf d)
 &\leq \exp\left(8\sum_{q\in Q}\frac1q\right)
 3^{\omega(d_1)+\omega(d_2)+\omega(d_3)}\\
 &\leq e^{8(T+1)}
 3^{\omega(d_1)+\omega(d_2)+\omega(d_3)}.
\end{align*}
Here we used
\[
 \sum_{q\in Q}\mathbf1_{q\mid d_j}\leq\omega(d_j)
 \qquad (1\leq j\leq3).
\]
Moreover, the sum of the corresponding weights is
\begin{align*}
 &\sum_{\mathbf d\in E^3\setminus D^3}
 3^{\omega(d_1)+\omega(d_2)+\omega(d_3)}
 (d_1d_2d_3)^{-1/3}\\
 &\hspace{35mm}=H_E^3-H_D^3\leq H^3-H_D^3.
\end{align*}
Proposition~\ref{prop:lattice} and \eqref{eq:D} imply
\begin{equation}\label{eq:middle}
 \sum_{\mathbf d\in E^3\setminus D^3}R_{\mathbf d}(X;a,M)
 \leq\delta\frac XM+O_{M,E}(X^{2/3}).
\end{equation}

It remains to remove the truncation.  For any fixed $\mathbf d$, the number
of $\mathbf x\in\N^3$ satisfying
$d_1x_1^3+d_2x_2^3+d_3x_3^3\leq X$ is at most
$X(d_1d_2d_3)^{-1/3}$.  Hence all triples with at least one 4-full part
outside $E$, even without a congruence condition, contribute at most
\begin{equation}\label{eq:tail}
 X(W^3-W_E^3)<\delta\frac XM.
\end{equation}
Combining \eqref{eq:small}, \eqref{eq:middle}, and \eqref{eq:tail}, the number
of ordered three-term representations in the progression is at most
\[
 3\delta\frac XM+O_{M,E}(X^{2/3}).
\]

The number of distinct represented values is no greater than the number of
ordered representations.  We conclude that
\[
 \limsup_{X\to\infty}\frac M X\R_3(X;a,M)
 \leq3\delta<\varepsilon,
\]
as required.
\end{proof}

\section{A remark on the counting argument}

We finish here by explaining why the elementary counting argument does not prove
the theorem.  Its order of magnitude is exactly $X$, as the following
calculation shows.

Let $\F_r$ be the set of positive $r$-full integers.  Every
$n\in\F_r$ has a unique expression
\begin{equation}\label{eq:r-part}
 n=x^r\prod_{s=1}^{r-1}b_s^{r+s},
\end{equation}
where the $b_s$ are squarefree and pairwise coprime.  If $\C_r$ denotes
the set of possible $(r+1)$-full parts in \eqref{eq:r-part}, then
\begin{equation}\label{eq:Wr}
W_r=\sum_{d\in\C_r}d^{-1/r}
=\prod_p\left(1+\sum_{s=1}^{r-1}p^{-1-s/r}\right)<\infty.
\end{equation}
Indeed, a prime whose exponent is congruent to $s\pmod r$, with
$1\leq s<r$, contributes one factor $p^{r+s}$ to the
$(r+1)$-full part; the remaining exponent is a non-negative multiple of
$r$.  This is the same prime-by-prime argument as in
Lemma~\ref{lem:canonical}, and it proves existence and uniqueness.

\begin{proposition}\label{prop:critical}
Let $T_r(X)$ be the number of ordered $r$-tuples
$(n_1,\ldots,n_r)\in\F_r^r$ with $n_1+\cdots+n_r\leq X$.  Then
\[
 T_r(X)\sim \Gamma(1+1/r)^rW_r^rX.
\]
The number of unordered multisets satisfying the same conditions is
asymptotic to
\[
 \frac{\Gamma(1+1/r)^r}{r!}W_r^rX.
\]
\end{proposition}

\begin{proof}
Fix a finite set $E\subset\C_r$.  For fixed parts
$d_1,\ldots,d_r\in E$, lattice-point counting gives
\[
 \#\{\mathbf x\in\N^r:d_1x_1^r+\cdots+d_rx_r^r\leq X\}
 \sim \Gamma(1+1/r)^rX(d_1\cdots d_r)^{-1/r}.
\]
The tuples having at least one $(r+1)$-full part outside $E$ are at most
\[
 X\bigl(W_r^r-W_{r,E}^r\bigr),
 \qquad W_{r,E}=\sum_{d\in E}d^{-1/r}.
\]
Letting $E$ exhaust $\C_r$ proves the ordered asymptotic.

By \eqref{eq:Wr}, the number of $r$-full integers up to $X$ is
$O_r(X^{1/r})$.  Ordered tuples with at least two equal entries therefore
contribute at most $O_r(X^{(r-1)/r})=o(X)$.  All other permutation orbits
have size $r!$, which proves the second assertion.
\end{proof}

Thus the elementary count is of order $X$, rather than $o(X)$.  In the
cubic proof, the saving appears only after local restrictions are imposed on
the 4-full parts.  Unfortunately, the argument does not give an equally effective
replacement when $r\geq4$.

\end{document}